\documentclass[12pt,oneside,english]{amsart}
\usepackage[latin9]{inputenc}
\usepackage{geometry}
\usepackage{amsthm}
\usepackage{amstext}
\usepackage{amssymb}
\usepackage[all]{xy}

\makeatletter
\numberwithin{equation}{section}
\numberwithin{figure}{section}
\theoremstyle{plain}
\newtheorem{thm}{\protect\theoremname}[section]
  \theoremstyle{plain}
  \newtheorem{lem}[thm]{\protect\lemmaname}
  \theoremstyle{definition}
  \newtheorem{defn}[thm]{\protect\definitionname}
  \theoremstyle{plain}
  \newtheorem{conjecture}[thm]{\protect\conjecturename}
  \theoremstyle{remark}
  \newtheorem{rem}[thm]{\protect\remarkname}
 \newcommand\thmsname{\protect\theoremname}
 \newcommand\nm@thmtype{theorem}
 \theoremstyle{plain}
 
 \newenvironment{namedthm}[1][Undefined Theorem Name]{
   \ifx{#1}{Undefined Theorem Name}\renewcommand\nm@thmtype{theorem*}
   \else\renewcommand\thmsname{#1}\renewcommand\nm@thmtype{namedtheorem}
   \fi
   \begin{\nm@thmtype}}
   {\end{\nm@thmtype}}
  \theoremstyle{plain}
  \newtheorem{prop}[thm]{\protect\propositionname}
  \theoremstyle{plain}
  \newtheorem{cor}[thm]{\protect\corollaryname}
  \theoremstyle{definition}
  \newtheorem{example}[thm]{\protect\examplename}


\makeatother

\usepackage{babel}
  \providecommand{\conjecturename}{Conjecture}
  \providecommand{\corollaryname}{Corollary}
  \providecommand{\definitionname}{Definition}
  \providecommand{\examplename}{Example}
  \providecommand{\lemmaname}{Lemma}
  \providecommand{\propositionname}{Proposition}
  \providecommand{\remarkname}{Remark}
  \providecommand{\theoremname}{Theorem}
\providecommand{\theoremname}{Theorem}

\begin{document}

\title{Densities of rational points and number fields}

\author{Takehiko Yasuda}
\begin{abstract}
We relate the problem of counting number fields, in particular, Malle's
conjecture with the problem of counting rational points on singular
Fano varieties, in particular, Batyrev and Tschinkel's generalization
of Manin's conjecture.
\end{abstract}

\address{Department of Mathematics, Graduate School of Science, Osaka University,
Toyonaka, Osaka 560-0043, Japan, tel:+81-6-6850-5326, fax:+81-6-6850-5327}

\email{takehikoyasuda@math.sci.osaka-u.ac.jp}

\maketitle
\global\long\def\AA{\mathbb{A}}
\global\long\def\PP{\mathbb{P}}
\global\long\def\NN{\mathbb{N}}
\global\long\def\GG{\mathbb{G}}
\global\long\def\ZZ{\mathbb{Z}}
\global\long\def\QQ{\mathbb{Q}}
\global\long\def\CC{\mathbb{C}}
\global\long\def\FF{\mathbb{F}}
\global\long\def\LL{\mathbb{L}}
\global\long\def\RR{\mathbb{R}}
\global\long\def\MM{\mathbb{M}}
\global\long\def\SS{\mathbb{S}}

\global\long\def\bx{\mathbf{x}}
\global\long\def\bf{\mathbf{f}}
\global\long\def\ba{\mathbf{a}}
\global\long\def\bs{\mathbf{s}}
\global\long\def\bt{\mathbf{t}}
\global\long\def\bw{\mathbf{w}}
\global\long\def\bb{\mathbf{b}}
\global\long\def\bv{\mathbf{v}}
\global\long\def\bp{\mathbf{p}}
\global\long\def\bq{\mathbf{q}}
\global\long\def\bm{\mathbf{m}}
\global\long\def\bj{\mathbf{j}}
\global\long\def\bM{\mathbf{M}}
\global\long\def\bd{\mathbf{d}}

\global\long\def\cN{\mathcal{N}}
\global\long\def\cW{\mathcal{W}}
\global\long\def\cY{\mathcal{Y}}
\global\long\def\cM{\mathcal{M}}
\global\long\def\cF{\mathcal{F}}
\global\long\def\cX{\mathcal{X}}
\global\long\def\cE{\mathcal{E}}
\global\long\def\cJ{\mathcal{J}}
\global\long\def\cO{\mathcal{O}}
\global\long\def\cD{\mathcal{D}}
\global\long\def\cZ{\mathcal{Z}}
\global\long\def\cR{\mathcal{R}}
\global\long\def\cC{\mathcal{C}}
\global\long\def\cL{\mathcal{L}}

\global\long\def\fs{\mathfrak{s}}
\global\long\def\fp{\mathfrak{p}}
\global\long\def\fm{\mathfrak{m}}
\global\long\def\fX{\mathfrak{X}}
\global\long\def\fV{\mathfrak{V}}
\global\long\def\fx{\mathfrak{x}}
\global\long\def\fv{\mathfrak{v}}
\global\long\def\fY{\mathfrak{Y}}

\global\long\def\rv{\mathbf{\mathrm{v}}}
\global\long\def\rx{\mathrm{x}}
\global\long\def\rw{\mathrm{w}}
\global\long\def\ry{\mathrm{y}}
\global\long\def\rz{\mathrm{z}}
\global\long\def\bv{\mathbf{v}}
\global\long\def\bx{\mathbf{x}}
\global\long\def\bw{\mathbf{w}}
\global\long\def\sv{\mathsf{v}}
\global\long\def\sx{\mathsf{x}}
\global\long\def\sw{\mathsf{w}}

\global\long\def\Spec{\mathrm{Spec}\,}
\global\long\def\Hom{\mathrm{Hom}}

\global\long\def\Var{\mathrm{Var}}
\global\long\def\Gal{\mathrm{Gal}}
\global\long\def\Jac{\mathrm{Jac}}
\global\long\def\Ker{\mathrm{Ker}}
\global\long\def\Image{\mathrm{Im}}
\global\long\def\Aut{\mathrm{Aut}}
\global\long\def\st{\mathrm{st}}
\global\long\def\diag{\mathrm{diag}}
\global\long\def\characteristic{\mathrm{char}}
\global\long\def\tors{\mathrm{tors}}
\global\long\def\sing{\mathrm{sing}}
\global\long\def\red{\mathrm{red}}
\global\long\def\Ind{\mathrm{Ind}}
\global\long\def\nr{\mathrm{nr}}
\global\long\def\ord{\mathrm{ord}}
\global\long\def\pt{\mathrm{pt}}
\global\long\def\op{\mathrm{op}}
\global\long\def\Val{\mathrm{Val}}
\global\long\def\Res{\mathrm{Res}}
\global\long\def\Pic{\mathrm{Pic}}
\global\long\def\disc{\mathrm{disc}}
 \global\long\def\length{\mathrm{length}}
\global\long\def\sm{\mathrm{sm}}
\global\long\def\top{\mathrm{top}}
\global\long\def\rank{\mathrm{rank}}
\global\long\def\Mot{\mathrm{Mot}}
\global\long\def\age{\mathrm{age}}
\global\long\def\et{\mathrm{et}}
\global\long\def\hom{\mathrm{hom}}
\global\long\def\tor{\mathrm{tor}}
\global\long\def\reg{\mathrm{reg}}
\global\long\def\cont{\mathrm{cont}}
\global\long\def\crep{\mathrm{crep}}
\global\long\def\Stab{\mathrm{Stab}}
\global\long\def\discrep{\mathrm{discrep}}
\global\long\def\crediv{\mathrm{crediv}}

\global\long\def\Conj#1{\mathrm{Conj}(#1)}
\global\long\def\KConj#1{K\mbox{-}\mathrm{Conj}(#1)}
\global\long\def\Mass#1{\mathrm{Mass}(#1)}
\global\long\def\Inn#1{\mathrm{Inn}(#1)}
\global\long\def\bConj#1{\mathbf{Conj}(#1)}
\global\long\def\Hilb{\mathrm{Hilb}}
\global\long\def\sep{\mathrm{sep}}
\global\long\def\GL#1#2{\mathrm{GL}_{#1}(#2)}
\global\long\def\codim{\mathrm{codim}}
\global\long\def\cd{\mathrm{cd}}

\global\long\def\Eta#1#2{#1\text{-}\mathrm{Eta}(#2)}
\global\long\def\Fie#1#2{#1\text{-}\mathrm{Fie}(#2)}
\global\long\def\etale#1#2{#1\text{-}\mathrm{eta}(#2)}
\global\long\def\ContHomGK#1{\Hom_{\cont}(\Gal(\bar{K}/K),#1)}
\global\long\def\ContSurGK#1{\mathrm{Sur}_{\cont}(\Gal(\bar{K}/K),#1)}
\global\long\def\fie#1#2{#1\text{-}\mathrm{fie}(#2)}
\global\long\def\ind{\mathrm{ind}}
\global\long\def\prim{\mathrm{prim}}
\global\long\def\sprim{\mathrm{sprim}}
\global\long\def\orig{\mathrm{orig}}
\global\long\def\vorig{\mathrm{vorig}}
\global\long\def\nor{\mathrm{nor}}
\global\long\def\Zeff{Z_{\mathrm{eff},n}(\PP^{m})}

\tableofcontents{}

\section{Introduction\label{sec:Introduction}}

In this paper, we relate two subjects in the number theory: the density
of rational points and the one of number fields.

As for the former, we consider an algebraic variety $X$ over a number
field $K$, and a suitable set $U\subset X(K)$ of $K$-points, obtained
by removing accumulating subsets\emph{ }from $X(K)$. It is expected
that the distribution of points in $U$ reflects the geometry of $X$.
We are interested in the case where $X$ is a Fano variety. Then $U$
tends to be an infinite set. Given a metric on the anti-canonical
sheaf $\omega_{X}^{-1}$, we can define a height function
\[
H:X(K)\to\RR_{>0}
\]
so that for each $B\in\RR_{>0}$, 
\[
N_{U}(B):=\sharp\{x\in U\mid H(x)\le B\}<\infty.
\]
We are interested in the asymptotic behavior of $N_{U}(B)$ as $B$
tends to infinity. Manin's conjecture \cite{Franke:1989go} concerns
this problem. A not so precise version states that 
\[
N_{U}(B)\sim c\cdot B\cdot(\log B)^{\rho-1},
\]
with $\rho$ the Picard number of $X$, under the assumption that
$X$ is smooth. Batyrev and Tschinkel \cite{MR1679843} generalized
it to Fano varieties having canonical singularities, where $\rho$
was replaced with a number incorporating singularities. 

As for the density of number fields, let $G$ be a transitive subgroup
of the symmetric group $S_{n}$. We consider degree $n$ extensions
$L/K$ of a given number field $K$ such that its Galois closure has
Galois group permutation isomorphic to $G$: we call such an extension
$L/K$ a \emph{$G$-field over $K$}. We put $M_{G,K}(B)$ to be the
number of $G$-fields with $|N_{K/\QQ}(D_{L/K})|\le B$, where $D_{L/K}$
is the discriminant of $L/K$ and $N_{K/\QQ}$ the norm of $K/\QQ$.
The asymptotic behavior of $M_{G,K}(B)$ as $B$ tends to infinity
is another concern of ours. Malle's conjecture \cite{MR1884706,MR2068887}
states that
\[
M_{G,K}(B)\sim c\cdot B^{1/\ind(G)}\cdot(\log B)^{\beta(G,K)-1}
\]
for some constant $c$ and invariants $\ind(G)$, $\beta(G,K)$ determined
by $G$ and $K$.

We would like to relates Malle's conjecture, and Batyrev and Tschinkel's
conjecture. Our strategy is to consider the quotient variety 
\[
X:=(\PP_{K}^{m})^{n}/G
\]
for the natural $G$-action on the power of the projective space.
If $m\cdot\ind(G)\ge2$, then $X$ is a Fano variety with only canonical
singularities. We exhibit a correspondence between \emph{primitive}
$K$-points of $X$ and \emph{original} $F$-points of $\PP_{K}^{m}$
with $F/K$ running over $G$-fields. Since the correspondence respect
heights, we obtain an equality among height zeta functions of the
form: 
\[
Z_{X(K)^{\prim}}(s)=\frac{1}{\sharp Z(G)}\cdot\sum_{L\in\Fie GK}Z_{\PP^{m}(L^{G_{1}})^{\orig}}(s),
\]
(for details, see Theorem \ref{thm: zeta equal permutation}). We
expect that the left side contains information about the density of
primitive $K$-points of $X$ and the right side contains information
about the density of $G$-fields. Using this equality together with
additional conjectures, we show partial results on implications between
Malle's conjecture, and Batyrev and Tschinkel's conjecture. 

Finally we briefly mention relations to other works. In the paper
\cite[page 153]{Ellenberg:2005bn} of Ellenberg and Venkatesh, it
was mentioned, as a comment by Tschinkel, a similarity between their
work on Malle's conjecture and Batyrev's one on rational points on
Fano varieties. In another paper of theirs \cite[page 732]{Ellenberg:2006js},
the relation between Malle's conjecture and Manin's conjecture was
more explicitly noted. In the same paper, they use the field of multi-symmetric
functions, which is the function field of the above quotient variety
$X$. The approach using $X$ or its function field is regarded as
a revisitation of Noether's approach to the inverse Galois problem
(see \cite{MR2363329}), incorporating newer materials such as Malle's
and Manin's conjectures. Our work is also a global analogue of the
wild McKay correspondence \cite{Wood-Yasuda-I,MR3230848,Yasuda:2013fk,Yasuda:2014fk},
which relates weighted counts of extensions of a local field with
stringy invariants of quotient varieties. 

Throughout the paper, $K$ denotes a number field, that is, a finite
extension of $\QQ$.

\subsection*{Acknowledgements}

I would like to thank Takashi Taniguchi, Takuya Yamauchi and Akihiko
Yukie for stimulating discussions. I also thank Seidai Yasuda for
reading the first draft and giving me many helpful comments.

\section{Rational points on singular Fano varieties\label{sec: Points-on-Fano-1}}

\subsection{Singularities\label{sec:Singularities}}

We first set up terminology on singularities. For details, we refer
the reader to \cite{MR3057950}. 

Let $X$ be a normal variety over $K$. We suppose that $X$ is $\QQ$-Gorenstein,
that is, the canonical divisor $K_{X}$ is $\QQ$-Cartier. A \emph{divisor
over $X$ }is a prime divisor on a normal modification $Y$ of $X$
($Y$ is a normal variety which is proper and birational over $X$).
Here we identify prime divisors on different modifications if they
give the same valuation of the function field. For a normal modification
$f:Y\to X$ with exceptional prime divisors $E_{i}$, we can uniquely
write 
\[
K_{Y}=f^{*}K_{X}+\sum_{i}a(E_{i})\cdot E_{i}\quad(a(E_{i})\in\QQ).
\]
The number $a(E_{i})$ is independent of the modification and it makes
sense to define $a(E)$ for a divisor $E$ over $X$: we call it the
\emph{discrepancy }of $E$. We call $E$ a \emph{crepant divisor }if
$a(E)=0.$ We define the \emph{minimal discrepancy }of $X$ by
\[
\discrep(X):=\inf_{E}a(E),
\]
where $E$ runs over all divisors over $X$. We say that $X$ is \emph{terminal
(resp. canonical) }or that $X$ has only\emph{ terminal (resp. canonical)}
singularities if 
\[
\discrep(X)>0\quad(\text{resp. }\ge0).
\]

\subsection{Heights\label{sec:Heights}}

Next, we recall the notion of heights of points. For details, we refer
the reader to \cite{MR2019019,MR2647601}.

Let $\Val(K)$ be the set of valuations (places) of $K$. For $v\in\Val(K)$,
we denote by $K_{v}$ the corresponding completion of $K$. If $p\in\Val(\QQ)$
is such that $v\mid p$ and $|\cdot|_{p}$ denotes the $p$-adic norm
on $\QQ_{p}$, then we define the \emph{$v$-adic norm }on $K_{v}$
by 
\[
|a|_{v}:=|N_{K_{v}/\QQ_{p}}(a)|_{p}.
\]

Let $X$ be a quasi-projective variety over $K$ and $\cL$ an invertible
sheaf on $X$. For $v\in\Val(K)$ and $x\in\cL(K_{v})$, the pullback
$x^{*}\cL$ to $\Spec K_{v}$ is regarded as a one-dimensional $K_{v}$-vector
space. We say that $\cL$ is \emph{(adelically) metrized }if $\cL$
is endowed with data of $v$-adic norms 
\[
\left\Vert \cdot\right\Vert _{v}:x^{*}\cL\to\RR_{\ge0}
\]
for all $v\in\Val(K)$ and $x\in\cL(K_{v})$ which satisfy some conditions
(for instance, see \cite{MR2019019,MR2647601}).

Given a metrized invertible sheaf $\cL$ on $X$, the \emph{height
}of a $K$-point $x\in X(K)$ is given by
\[
H_{\cL}(x):=\prod_{v\in\Val(K)}\left\Vert s\right\Vert _{v}^{-1}\in\RR_{>0}
\]
for any $0\ne s\in x^{*}\cL$. Given two metrized invertible sheaves
$\cL_{1}$ and $\cL_{2}$, then we can naturally metrize the tensor
product $\cL_{1}\otimes\cL_{2}$, and the associated height function
satisfies
\[
H_{\cL_{1}\otimes\cL_{2}}(x)=H_{\cL_{1}}(x)\cdot H_{\cL_{2}}(x).
\]

We say that a Cartier divisor $D$ on $X$ is \emph{metrized }if $\cO_{X}(D)$
is metrized. For a metrized Cartier divisor $D$, we define heights
by 
\[
H_{D}(x):=H_{\cO_{X}(D)}(x).
\]
We say that a $\QQ$-Cartier (Weil) divisor $D$ on $X$ is \emph{metrized
}if $mD$ is metrized, where $m$ is the least positive integer with
$mD$ Cartier. For a metrized $\QQ$-Cartier divisor $D$, we put
\[
H_{D}(x):=H_{mD}(x)^{1/m}
\]
with $m$ as above. We often write $H_{D}(x)$ simply as $H(x)$,
when $D$ is understood. 

For a finite field extension $L/K$, we put $X_{L}:=X\otimes_{K}L$
and $D_{L}$ to be the pullback of $D$ to $X_{L}$. For $x\in X(L)$,
regarding $x$ as an element of $X_{L}(L)$, we define
\[
H_{D}(x):=H_{D_{L}}(x).
\]
Basic properties of heights are as follows.
\begin{lem}
\label{lem: properties height}Let $L/K$ be a finite field extension.
\begin{enumerate}
\item For two (metrized) $\QQ$-Cartier divisors $D$ and $D'$ and for
$x\in X(L)$, we have $H_{D+D'}(x)=H_{D}(x)H_{D'}(x)$.
\item For a morphism $f:Y\to X$ of $K$-varieties and $y\in Y(L)$, we
have $H_{f^{*}D}(y)=H_{D}(f(y))$.
\item For an extension $M/L$ and for $x\in X(L)$, if $x_{M}$ denotes
the composition $\Spec M\to\Spec L\xrightarrow{x}X$, then 
\[
H_{D}(x_{M})=H_{D}(x)^{[M:L]}.
\]

\end{enumerate}
\end{lem}
\begin{proof}
All these properties are direct consequences of known properties of
heights (for instance, see \cite{MR2647601}) and the definitions
above.
\end{proof}

\subsection{Rational points on singular Fano varieties\label{sec: Points-on-Fano}}

We review a generalization of Manin's conjecture by Batyrev and Tschinkel
\cite{MR1679843} and Beukers' observation, which concern the density
of rational points on singular Fano varieties.
\begin{defn}
A projective variety over $K$ is called a \emph{terminal (resp. canonical)
Fano variety} if it is normal, $\QQ$-Gorenstein and terminal (resp.
canonical), and its anti-canonical divisor $-K_{X}$ is ample. 
\end{defn}
Let $X$ be a canonical Fano variety. We suppose that $-K_{X}$ is
metrized and $X$ is given the height function $H$ associated to
it. For a subset $U\subset X(K),$ we put 
\begin{align*}
N_{U}(B): & =\sharp\{x\in U\mid H(x)\le B\}.
\end{align*}
To state a conjecture on the asymptotic behavior of $N_{U}(B)$, we
need two invariants of $X$. Firstly we consider the Picard number
$\rho(X)$ of $X$, that is, the rank of the Néron-Severi group. Secondly
we put $\cd(X)$ to be the number of crepant divisors over $X$. Note
that these invariants may change after an extension of the base field.
The following is a modification of a ``conjecture'' by Batyrev and
Tschinkel \cite[page 323]{MR1679843}. 
\begin{conjecture}
\label{conj: Manin variant}If $U$ is a suitable subset of $X$,
then for some positive constant $c$, 
\[
N_{U}(B)\sim c\cdot B\cdot(\log B)^{\rho(X)+\cd(X)-1}\quad(B\to\infty).
\]
\end{conjecture}
\begin{rem}
\label{rem: Manin variant}
\begin{enumerate}
\item In \cite{MR1679843}, the authors put more specific assumptions on
$X$ and $U$. We note that they did not call it a conjecture. Actually
we should regard the conjecture above as rather an optimistic expectation.
For this reason, we left the inaccuracy on what $U$ is, and did not
mention the necessity of a finite extension of the base field.
\item The exponent of $B$ was denoted by $\alpha_{\cL}(V)$ in \cite{MR1679843}.
In our situation, it is equal to one, since $\cL=\omega_{X}^{-1}$,
although it is not generally an invertible sheaf but only a $\QQ$-invertible
sheaf, strictly speaking. 
\item The exponent of $\log B$ was denoted by $\beta_{\cL}(V)-1$ in \cite{MR1679843}.
In our situation, this number is given as follows: let $f:Y\to X$
be a resolution of singularities and $l$ the number of the discrepant
(= not crepant) prime exceptional divisors of $f$. Then $\beta_{\cL}(V)$
is the Picard number of $Y$ minus $l$. The Picard number of $Y$
is the Picard number of $X$ plus the number of all exceptional prime
divisors over $X$. Therefore, eventually, $\beta_{\cL}(V)$ is the
Picard number of $X$ plus the number of crepant divisors over $X$.
\end{enumerate}
\end{rem}
On the constant $c$ in the asymptotic formula in the conjecture,
we mention the following observation by Beukers (see \cite[page 324]{MR2019019}).
\begin{namedthm}[Observation]
In some examples where $X$ is defined over the ring of integers
$\cO_{K}$, the constant $c$ contains, as a factor, the Euler product
\[
\prod_{v\in\Val(K)_{f}}\left(1-\frac{1}{N(v)}\right)^{\rho(X)}\frac{\sharp X_{R(v)}(R(v))}{N(v)^{\dim X}},
\]
where $\Val(K)_{f}$ is the set of non-archimedian valuations (finite
places), $R(v)$ the residue field of $K_{v}$ and $N(v)$ the cardinality
of $R(v)$. 
\end{namedthm}
Perhaps motivated by this, Peyre \cite{MR1340296} developed a conjecture
on the value of the constant $c$ by means of Tamagawa measures. Batyrev
and Tschinkel \cite{MR1679843} then generalized it further. Although
their works seem closely related to ours, we do not pursue it in this
paper.

\subsection{Projective spaces\label{subsec:Projective-spaces}}

Now let us focus on the simplest Fano variety, that is, the projective
space. Understanding this case well is crucial in our applications.
Fortunately there is a precise result by Schanuel \cite{MR0162787}. 

Let $\PP^{m}=\PP_{K}^{m}$ be the projective $m$-space over $K$.
For a suitable metric on $\cO_{\PP^{m}}(1)$, we have 
\[
H_{\cO(1)}((x_{0}:\cdots:x_{n}))=\prod_{v\in\Val(K)}\sup_{i}\left\Vert x_{i}\right\Vert _{v}
\]
(see \cite{MR2647601}). We give a metric to the anti-canonical divisor
$-K_{\PP^{m}}$ induced from the isomorphism $\cO(-K_{\PP^{m}})=\cO(m+1)$,
and suppose that $\PP^{m}$ is given the height function $H=H_{-K_{\PP^{m}}}$
induced from $-K_{\PP^{m}}$ metrized in this way. Namely $H$ is
given by
\[
H(x)=\prod_{v\in\Val(K)}\sup_{i}\left\Vert x_{i}\right\Vert _{v}^{m+1}.
\]

\begin{thm}[\cite{MR0162787}]
\label{thm:Schanuel}We have
\[
N_{\PP^{m}(K)}(B)\sim\delta_{K,m}\cdot B.
\]
Here we put 
\begin{equation}
\delta_{K,m}:=\frac{(2^{r_{1}}(2\pi)^{r_{2}})^{m+1}\cdot(m+1)^{r_{1}+r_{2}-1}\cdot h\cdot R}{\zeta_{K}(m+1)\cdot w\cdot d^{(m+1)/2}},\label{eq: delta def}
\end{equation}
following the standard notation for invariants of $K$: $r_{1}$ and
$r_{2}$ are the numbers of real and complex archimedian valuations,
$w$ the number of roots of unity, $h$ the class number, $d$ the
absolute value of the absolute discriminant, $R$ the regulator and
$\zeta_{K}$ the Dedekind zeta function. 
\end{thm}
We also write $d$ as $d_{K}$, to specify the field $K$. We sometimes
consider $\delta_{K,m}$ as an approximation of the power $d^{-m/2}$
of the discriminant and use it as a \emph{weight} of $K$, when counting
number fields. This is justified by the following result, a consequence
of the Siegel-Brauer theorem \cite{Brauer:1947vo}.
\begin{prop}
\label{prop:Siegel-Brauer}Fix $n$ and $m$. When $K$ varies among
degree $n$ extensions of $\QQ$, then 
\[
\log\delta_{K,m}\sim\log(d_{K}^{-m/2})\quad(d_{K}\to\infty).
\]
In particular, for every positive number $\epsilon$, 
\[
d_{K}^{-m/2-\epsilon}\ll\delta_{K,m}\ll d_{K}^{-m/2+\epsilon}.
\]
\end{prop}
\begin{proof}
The number 
\[
\frac{(2^{r_{1}}(2\pi)^{r_{2}})^{m+1}(m+1)^{r_{1}+r_{2}-1}}{\zeta_{K}(m+1)w}
\]
is bounded from above and below. For instance, concerning the zeta
value $\zeta_{K}(m+1)$, we generally have that for $s>1$, 
\[
\zeta(ns)=\prod_{p}(1-p^{-ns})^{-1}\le\zeta_{K}(s)\le\prod_{p}(1-p^{-1})^{-n}=\zeta(s)^{n}.
\]
Here $p$ runs over the prime numbers and $\zeta(s)$ is the Riemann
zeta functions.

As for $w$, let $L$ be the largest cyclotomic field in $K$. Putting
$\phi$ as Euler's totient function, we have 
\[
[L:\QQ]=\phi(w)\le n.
\]
The boundedness of $w$ is now followed from the finiteness of $\bigcup_{l\le n}\phi^{-1}(l)$
(see \cite{Gupta:1981tv}). The first assertion follows from the Siegel-Brauer
theorem \cite[Theorem 2]{Brauer:1947vo}: when the degree of $K/\QQ$
is fixed, we have 
\[
\log(hR)\sim\log\sqrt{d}\quad(d\to+\infty).
\]

For the second assertion, for every $\alpha>1$, there exists a positive
integer $n$ such that for every $K$ with $d_{K}\ge n$, 
\[
\alpha\cdot\log d_{K}^{-m/2}\le\log\delta_{K,m},
\]
and 
\[
d_{K}^{-\alpha m/2}\le\delta_{K,m}.
\]
This shows 
\[
d_{K}^{-\alpha m/2}\ll\delta_{K,m}.
\]
Similarly, for $\beta<1$, 
\[
\delta_{K,m}\ll d_{K}^{-\beta m/2}.
\]
We have proved the second assertion.
\end{proof}

\section{The density of number fields}

In this section, we briefly review Malle's conjecture and Bhargava's
conjecture on the density of number fields. For details, we refer
the reader to \cite{Belabas:2005tc}.

Let $n$ be an integer $\ge2$ and $G$ a transitive subgroup of $S_{n}$. 
\begin{defn}
A \emph{$G$-field (over $K$) }is a field extension $L/K$ of degree
$n$ such that the Galois group $\Gal(\hat{L}/K)$ of the Galois closure
$\hat{L}/K$, acting on the set of $K$-embeddings $\hat{L}\hookrightarrow\CC$,
is permutation isomorphic to $G$. Two $G$-fields are said to be
\emph{isomorphic} if they are isomorphic as $K$-algebras. The set
of isomorphism classes is denoted by $\fie GK$. (We will call a $G$-field
a \emph{small }$G$-field in later sections to distinguish it from
its Galois closure.)
\end{defn}
For a positive real number $B$, we put
\begin{align*}
M_{G,K}(B): & =\sharp\{L\in\fie GK\mid|N_{K/\QQ}(D_{L/K})|\le B\},
\end{align*}
where $D_{L/K}$ is the discriminant of $L/K$. We are interested
in the asymptotic behavior of $M_{G,K}(B)$ as $B$ tends to infinity.
To state Malle's conjecture, we define two invariants. We put
\[
[n]:=\{1,2,\dots,n\},
\]
which has the natural $G$-action.
\begin{defn}
We define the \emph{index }of $g\in G$ by
\[
\ind(g):=n-\sharp\{g\text{-orbits in }[n]\}.
\]
We put
\begin{align*}
\ind(G): & =\min\{\ind(g)\mid1\ne g\in G\}.
\end{align*}

\end{defn}
Denoting the algebraic closure of $K$ by $\bar{K}$, we consider
the natural $\Gal(\bar{K}/K)$-action on the set $\Conj G$ of the
conjugacy classes of $G$ (see \cite[12.4]{MR0450380}). We denote
by $\KConj G$ the quotient of this action and call elements of $\KConj G$
\emph{$K$-conjugacy classes} of $G$.
\begin{defn}
We put 
\begin{align*}
\beta(G,K): & =\sharp\left\{ [g]\in\KConj G\mid\ind(g)=\ind(G)\right\} .
\end{align*}

\end{defn}
Malle's conjecture is as follows.
\begin{conjecture}[\cite{MR1884706,MR2068887}]
\label{conj: Malle}We have
\[
M_{G,K}(B)\sim c\cdot B^{1/\ind(G)}(\log B)^{\beta(G,K)-1}.
\]

\end{conjecture}
Actually there exists a counterexample to this conjecture \cite{Kluners:2005cg}.
One possible way to rescue the situation would be to ask whether the
conjecture holds if one replace $K$ with a sufficiently large finite
extension. This is related to replacing the base field by a finite
extension in variants of Manin's conjecture (see Remark \ref{rem: Manin variant}). 

The index, $\ind(G)$, is equal to one if and only if $G$ contains
a transposition. If it is the case, then $\beta(G,K)=1$ \cite[Lemma 2.2]{MR2068887}.
For instance, for $G=S_{n}$, the conjecture gives
\[
M_{S_{n},K}(B)\sim c\cdot B.
\]
which is a well known conjecture. Moreover, in this case, Bhargava
conjectured the precise value of the constant $c$. According to \cite[Conjecture 6.3]{Belabas:2005tc}
and Bhargava's mass formula for étale extensions of a local field
\cite{MR2354798}, we can write it as follows.
\begin{conjecture}
\label{conj: Bhargava}We have
\[
M_{S_{n},K}(B)\sim c\cdot B,
\]
where 
\[
c=\frac{1}{2}\cdot\Res_{s=1}\zeta_{K}(s)\cdot\prod_{v\in\Val(K)_{f}}\left(1-\frac{1}{N(v)}\right)\left(\sum_{i=0}^{n-1}P(n,n-i)N(v)^{-i}\right)
\]
and $P(n,n-i)$ is the number of partitions of $n$ by exactly $n-i$
parts.
\end{conjecture}

\section{Comparison of constants}

In this section, we relate constants appearing in Malle's and Bhargava's
conjectures with the geometry of a quotient variety $X_{G,m}$ defined
as follows. Let $Y_{m}:=(\PP_{K}^{m})^{n}$, the $n$-th power of
the projective $m$-space over $K$. For a transitive subgroup $G\subset S_{n}$,
we put $X_{G,m}$ to be the quotient variety $Y_{m}/G$ by the natural
$G$-action on $Y_{m}$. We write the associated Galois cover as
\begin{equation}
\pi_{G,m}:Y_{m}\to X_{G,m}.\label{eq: main Galois cover}
\end{equation}

\subsection{Indices vs.\ discrepancies}
\begin{defn}
Let $1\ne G\subset\GL n{\CC}$ be a non-trivial finite subgroup. Diagonalizing
each element $g\in G$, we can write 
\[
h^{-1}gh=\diag(\zeta^{a_{1}},\dots,\zeta^{a_{n}}),
\]
where $h\in\GL n{\CC}$, $\zeta=\exp(2\pi\sqrt{-1}/l)$ and $0\le a_{i}<l$.
We define the \emph{age} of $g$ by
\[
\age(g):=\frac{1}{l}\sum_{i=1}^{n}a_{i}\in\QQ
\]
and the \emph{age }of $G$ by
\[
\age(G):=\min\{\age(g)\mid1\ne g\in G\}.
\]
We say that $1\ne g\in G$ is a \emph{pseudo-reflection }if the fixed
point locus $(\CC^{n})^{g}$ has codimension one. 
\end{defn}
This representation-theoretic invariant, age, determines the discrepancy
of the associated quotient variety:
\begin{prop}[\cite{MR2271984}]
\label{prop: discrep age}Suppose that $G$ has no pseudo-reflection.
Then
\[
\discrep(\CC^{n}/G)=\min\{\age(g)\mid1\ne g\in G\}-1.
\]

\end{prop}
For each $n$, we regard $S_{n}$ as a subgroup of $\GL n{\CC}$ by
the standard permutation representation $S_{n}\hookrightarrow\GL n{\CC}$. 
\begin{lem}[cf.\ \cite{Wood-Yasuda-I}]
\label{lem: age index}For $g\in S_{n}\subset\GL n{\CC}$, 
\[
2\cdot\age(g)=\ind(g).
\]
\end{lem}
\begin{proof}
From the additivity of age and index, we may suppose that $g$ is
the cyclic permutation, $1\mapsto2\mapsto\cdots\mapsto n\mapsto1$.
Then a diagonalization of $g$ is
\[
\diag(1,\zeta,\dots,\zeta^{n-1})
\]
with $\zeta=\exp(2\pi\sqrt{-1}/n)$. We have
\[
\age(g)=\frac{1}{n}\sum_{i=0}^{n-1}i=\frac{n-1}{2}=\frac{\ind(g)}{2}.
\]
\end{proof}
\begin{prop}
Suppose that $m\cdot\ind(G)\ge2$. Then $\pi_{G,m}$ is étale in codimension
one and
\[
\discrep(X_{G,m})=\frac{m\cdot\ind(G)}{2}-1.
\]
 In particular, $X_{G,m}$ is canonical. \end{prop}
\begin{proof}
If $\ind(G)\ge2$, then $G$ does not contain a transposition. If
$\ind(G)=1$, then $m\ge2$. It follows that $\pi_{G,m}$ is étale
in codimension one. If $X_{G,m}=\bigcup U_{i}$ is an open cover,
then 
\[
\discrep(X_{G,m})=\min_{i}\discrep(U_{i}).
\]
Moreover $\discrep(X)$ is stable under the extension of base field.
Therefore we may replace $Y_{m}$ with $(\CC^{m})^{n}$. Then the
proposition follows from Proposition \ref{prop: discrep age} and
Lemma \ref{lem: age index}. 
\end{proof}

\subsection{$K$-conjugacy classes of minimal index vs.\ crepant divisors}

Next we interpret the constant $\beta(G,K)$ in Malle's conjecture
in terms of the geometry of $X_{G,m}$.
\begin{prop}
If $m\cdot\ind(G)\ge2$, then $\beta(G,K)$ is equal to the number
of divisors $E$ over $X_{G,m}$ with $a(E)=\discrep(X_{G,m})$.\end{prop}
\begin{proof}
The proof is similar to the one of \cite[Theorem 5.4]{Wood-Yasuda-I}.
We only sketch the outline. Let $X:=X_{G,m}$ and $X_{L}:=X\otimes_{K}L$
for a field extension $L/K$. Thanks to Lemma \ref{lem: age index},
the McKay correspondence \cite{MR1677693} states that for a suitable
extension $L/K$, divisors $E$ over $X_{L}:=X\otimes_{K}L$ with
\[
a(E)=\discrep(X_{L})=\discrep(X)
\]
correspond to conjugacy classes $[g]\in\Conj G$ with $\ind(g)=\ind(G)$.
We denote by $A$ the set of such divisors over $X_{L}$. As such
an extension $L/K$, we can take, for instance, the cyclotomic extension
$K(\zeta_{l})/K$, where $l=\sharp G$ and $\zeta_{l}$ is a primitive
root of unity. The set of divisors $E$ over $X$ with $a(E)=\discrep(X)$
is identified with the $\Gal(K(\zeta_{l})/K)$-quotient of $A$. The
correspondence above of $A$ with a set of conjugacy classes is equivariant
with respect to actions of $\Gal(K(\zeta_{l})/K)$. We can see this,
for instance, by using twisted jets \cite{MR2271984}. This proves
the proposition.
\end{proof}
We would like to specialize Conjecture \ref{conj: Manin variant}
to the case where $X=X_{G,m}$. Since $X$ now has Picard number one,
the conjecture gives
\[
N_{U}(B)\sim c\cdot B\cdot(\log B)^{\cd(X)}.
\]
Suppose now that $m\cdot\ind(G)=2$. For instance, it is the case
when $G=S_{n}$ and $m=2$ or when $G=A_{n}$ and $m=1$. From the
proposition above, Malle's conjecture gives
\[
M_{G,K}(B)\sim c\cdot B^{m/2}\cdot(\log B)^{\cd(X)}.
\]

\subsection{Euler products}

Concerning constants $c$ in asymptotic formulas of the form 
\[
c\cdot B^{a}\cdot(\log B)^{b},
\]
we mentioned Beukers' observation on the density of rational points
and Bhargava's conjecture on the density of $S_{n}$-fields. Euler
products in their statements are related as follows. 

Suppose $G=S_{n}$ and $m=2$. Then $X=X_{S_{n},2}$ is nothing but
the symmetric product $S^{n}\PP^{2}$. Let $Z$ be the Hilbert scheme
of $n$ points on $\PP^{2}$. Then the Hilbert-Chow morphism $Z\to X$
is a crepant resolution (the pullback of $K_{X}$ coincides with $K_{Z}$),
as proved by Beauville \cite{0537.53056}. Let $X^{\circ}:=S^{n}\AA^{2}$,
the symmetric product of the affine plane, and $Z^{\circ}$ the Hilbert
scheme of $n$ points on $\AA^{2}$. They are open subvarieties of
$X$ and $Z$ respectively, and the induced map $Z^{\circ}\to X^{\circ}$
is a crepant resolution, too.  In fact, the Hilbert scheme is naturally
defined over the ring of integers $\cO_{K}$. Hence we can take the
reduction $Z_{R(v)}^{\circ}$ for $v\in\Val(K)_{f}$. For instance,
from \cite[Corollary 3.1]{MR2492446}, we have

\[
\sharp Z_{R(v)}^{\circ}(R(v))=\sum_{i=1}^{n}P(n,i)N(v)^{n+i}=\sum_{i=0}^{n-1}P(n,n-i)N(v){}^{2n-i}.
\]
Therefore, the Euler product appearing in Conjecture \ref{conj: Bhargava}
is rewritten as
\[
\prod_{v\in\Val(K)_{f}}\left(1-\frac{1}{N(v)}\right)^{\rho(X)}\frac{\sharp Z_{R(v)}^{\circ}(R(v))}{N(v)^{\dim Z}},
\]
having the apparent similarity with the Euler product in Beukers'
observation. This coincidence seems more mysterious to the author
than the ones for constants $a$ and $b$. He does not know whether
discussions in later sections by height zeta functions help us to
clarify anything behind (see Remark \ref{rem: m ind =00003D2}). 
\begin{rem}
A relation between Bhargava's mass formula \cite{MR2354798} and the
Hilbert scheme of points was first observed in \cite{Wood-Yasuda-I}.
\end{rem}

\section{Galois covers and correspondences of points\label{sec:Galois-covers}}

In this section, we show a correspondence between points of $X_{G,m}$
and points of $\PP_{K}^{m}$ in a slightly generalized situation.

\subsection{General Galois covers\label{sub:General-Galois-covers}}
\begin{defn}
A \emph{large $G$-field} is a Galois extension $L/K$ endowed with
an isomorphism $\Gal(L/K)\cong G$. Two large $G$-fields $L$ and
$L'$ are said to be\emph{ isomorphic }if there exists a $G$-equivariant
$K$-isomorphism $L\cong L'$. The set of isomorphism classes is denoted
by $\Fie GK$. 
\end{defn}

Let $G$ be a finite group, $Y$ a $K$-variety endowed with a faithful
$G$-action and $X:=Y/G$ the quotient variety with the quotient morphism
$\pi:Y\to X$.
\begin{defn}
\label{def: equivariant points-1}For $L\in\Fie GK$, an $L$-point
$y\in Y(L)$ is called \emph{$G$-equivariant }if the morphism $y:\Spec L\to Y$
is $G$-equivariant. We denote by $Y(L)^{G}$ the set of $G$-equivariant
$L$-points. 
\end{defn}
Note that $Y(L)^{G}$ is the fixed point locus of the $G$-action
on $Y(L)$ given by $g\cdot y:=g\circ y\circ g^{-1}$. We put 
\[
Y\{K\}:=\bigsqcup_{L\in\Fie GK}\frac{Y(L)^{G}}{Z(G)},
\]
where $Z(G)$ is the center of $G$ and identified with the set of
$G$-equivariant $K$-automorphisms of $L$. Let $Y^{\circ}\subset Y$
and $X^{\circ}\subset X$ be the largest open subvarieties where $\pi$
is unramified. We similarly put
\[
Y^{\circ}\{K\}:=\bigsqcup_{L\in\Fie GK}\frac{Y^{\circ}(L)^{G}}{Z(G)}\subset Y\{K\}.
\]

For a $G$-equivariant point $y:\Spec L\to Y$, the induced morphism
between the $G$-quotients of the source and target gives a $K$-point
of $X$, and defines a map
\[
\pi_{*}:Y\{K\}\to X(K).
\]

\begin{defn}
We define the set of \emph{primitive $K$-points }of $X$ by 
\[
X(K)^{\prim}:=X^{\circ}(K)\setminus\bigcup_{H\subsetneq G}\Image((Y^{\circ}/H)(K)\to X^{\circ}(K)).
\]
\end{defn}
\begin{rem}
\label{rem: thin primitive}For a $K$-variety $S$, a subset $A\subset S(K)$
is called \emph{thin }if $A$ is contained in the image of $T(K)\to S(K)$
for a generically finite morphism $T\to S$ which does not admit a
rational section. The set $X(K)\setminus X(K)^{\prim}$ is a thin
subset of $X(K)$. Colliot-Thélène conjectured that if $S$ is unirational,
then $S(K)$ is not thin (see \cite[page 30]{MR2363329}). The conjecture
implies the folklore conjecture that every finite group is the Galois
group of some Galois extension of $\QQ$. His conjecture also shows
that if $Y$ is rational, then $X(K)^{\prim}$ is not empty. \end{rem}
\begin{lem}
A $K$-point $x\in X(K)$ is primitive if and only if the fiber product
\[
\Spec K\times_{x,X,\pi}Y
\]
is connected and étale over $\Spec K$.\end{lem}
\begin{proof}
It is easy to see that $x\in X^{\circ}(K)$ if and only if $ $$\Spec K\times_{x,X,\pi}Y$
is étale over $\Spec K$. Suppose that $x$ is not primitive. Then
there exists a subgroup $H\subsetneq G$ such that $x$ lifts to a
$K$-point of $Y^{\circ}/H$. It follows that $\Spec K\times_{x,X}(Y/H)$
is the disjoint union of $[G:H]$ copies of $\Spec K$. In particular,
it is not connected. Hence $ $$\Spec K\times_{x,X,\pi}Y$ is not
connected, either. 

Conversely suppose that $ $$\Spec K\times_{x,X,\pi}Y$ is not connected.
Let $H$ be the stabilizer of a connected component. Then there is
a lift of $x$ in $Y/H(K)$. Hence $x$ is not primitive.\end{proof}
\begin{prop}
The map $\pi_{*}$ induces a bijection
\[
Y^{\circ}\{K\}\to X(K)^{\prim}.
\]
\end{prop}
\begin{proof}
Let $(y:\Spec L\to Y)\in Y^{\circ}\{K\}$ and $x:=\pi_{*}(y)\in X(K)$.
Then there exists a natural $G$-equivariant $K$-morphism 
\[
\Spec L\to\Spec K\times_{x,X,\pi}Y.
\]
Since the source and the target are both étale $G$-torsors over $\Spec K$,
this morphism is an isomorphism. Thus the isomorphism class of $y$
is reconstructed from $x$. This shows that the map $\pi_{*}|_{Y^{\circ}\{K\}}$
is an injection into $X(K)^{\prim}$. 

Conversely, we start with an arbitrary primitive point $x\in X(K)^{\prim}$.
The last lemma shows that $\Spec K\times_{x,X,\pi}Y$ is isomorphic
to $\Spec L$ for some $L\in\Fie GK$. The second projection $\Spec K\times_{x,X,\pi}Y\to Y$
defines an equivariant point $\Spec L\to Y^{\circ}$, which is a lift
of $x$. This shows that $\pi_{*}:Y^{\circ}\{K\}\to X(K)^{\prim}$
is surjective. \end{proof}
\begin{prop}
\label{prop:H(y)=00003DH(pi(y))}Let $D$ be a metrized divisor $X$.
For $y\in Y\{K\}$, we have
\[
H_{\pi^{*}D}(y)^{1/\sharp G}=H_{D}(\pi_{*}(y)).
\]
\end{prop}
\begin{proof}
From Lemma \ref{lem: properties height}, for $y\in Y(L)^{G}$, 
\begin{align*}
H_{\pi^{*}D}(y) & =H_{D}(\Spec L\xrightarrow{y}\Spec Y\to\Spec X)\\
 & =H_{D}(\Spec L\to\Spec K\xrightarrow{\pi_{*}(y)}\Spec X)\\
 & =H_{D}(\pi_{*}(y))^{[L:K]}.
\end{align*}
This shows the assertion.
\end{proof}

\subsection{Permutation actions of a transitive subgroup\textmd{\label{subsec:permutation-actions}}}

Now suppose that $G$ is a transitive subgroup of $S_{n}$, acting
on $[n]:=\{1,2,\dots,n\}$. For each $i\in[n]$, we put $G_{i}\subset G$
to be the stabilizer subgroup of $i$. Since $G_{i}$, $i\in[n]$
are conjugate one another, for each $L\in\Fie GK$, the invariant
subfields $L^{G_{i}}$, $i\in[n]$ are isomorphic over $K$ one another.
There exists a map
\[
\phi:\Fie GK\to\fie GK,\, L\mapsto L^{G_{1}}.
\]
The group of $K$-automorphisms of a small $G$-field is isomorphic
to the opposite group of the centralizer $C_{S_{n}}(G)$.

\begin{lem}
\label{lem: phi-1}The map $\phi$ is a $\frac{\sharp N_{S_{n}}(G)\sharp Z(G)}{\sharp C_{S_{n}}(G)\sharp G}$-to-one
surjection.\end{lem}
\begin{proof}
The map is clearly surjective. Let $\ContSurGK G$ be the set of continuous
surjections of $\Gal(\bar{K}/K)$ to $G$. The natural map
\[
\ContSurGK G\to\Fie GK
\]
can be identified with the quotient map associated to the $G$-action
on $\ContSurGK G$ by conjugation. Therefore this map is a $\frac{\sharp G}{\sharp Z(G)}$-to-one
surjection. On the other hand, the natural map 
\begin{equation}
\ContSurGK G\to\fie GK\label{eq: Sur to G-fie-1}
\end{equation}
is identified with the restriction of the natural map
\[
\ContHomGK{S_{n}}\to\fie nK,
\]
where $\fie nK$ is the set of isomorphism classes of degree $n$
extensions $L/K$. The last map is, in turn, identified with the quotient
map associated to the $S_{n}$-action. Hence map (\ref{eq: Sur to G-fie-1})
is a $\frac{\sharp N_{S_{n}}(G)}{\sharp C_{S_{n}}(G)}$-to-one surjection.
We have proved the lemma.
\end{proof}
From the lemma, counting large $G$-fields and counting small $G$-fields
are equivalent problems. 

Let $W$ be a $K$-variety and let $Y:=W^{n}$, which has a natural
$G$-action. For $L\in\Fie GK$, $y\in Y(L)^{G}$ and $i\in[n]$,
the composition 
\[
\Spec L\xrightarrow{y}W^{n}\xrightarrow{p_{i}}W
\]
is stable under the $G_{i}$-action on the source, where $p_{i}$
is the $i$-th projection. Therefore we obtain a morphism
\[
\psi_{i}(y):\Spec L^{G_{i}}\to W.
\]

\begin{lem}
For each $i\in[n]$, the induced map
\begin{align}
\psi_{i}:Y(L)^{G} & \to W\left(L^{G_{i}}\right),\, y\mapsto\psi_{i}(y)\label{eq:Y(L)G --> W( )}
\end{align}
is bijective.\end{lem}
\begin{proof}
The given point $y\in Y(L)^{G}$ is be reconstructed from the induced
point $w=\psi_{i}(y)\in W\left(L^{G_{i}}\right)$. Indeed, if $\sigma_{j}\in G$
is any element sending $i$ to $j$. then
\begin{equation}
y=\prod_{i=1}^{n}(\Spec L\xrightarrow{\sigma_{j}}\Spec L\xrightarrow{\alpha}\Spec L^{G_{i}}\xrightarrow{w}W),\label{eq:y as product}
\end{equation}
where $\alpha$ is the morphism corresponding to the inclusion $L^{G_{i}}\subset L$.
This proves the lemma.
\end{proof}
For $i,j\in[n]$, let 
\[
\alpha_{ij}:W(L^{G_{i}})\to W(L^{G_{j}})
\]
be the unique map such that $\alpha_{ij}\circ\psi_{i}=\psi_{j}$.
Any element $g\in G=\Gal(L/K)$ with $g(j)=i$ gives an isomorphism
$L^{G_{i}}\to L^{G_{j}}$, being independent of the choice of $g$.
This isomorphism induces the map $\alpha_{ij}$. 
\begin{defn}
\label{def: original}With notation as above, we say that a point
$x\in W(L^{G_{i}})$ is \emph{original }if for any nontrivial normal
subgroup $1\ne H\vartriangleleft G$, $x\notin W(L^{\left\langle G_{i},H\right\rangle })$.
Here we identify $W(L^{\left\langle G_{i},H\right\rangle })$ as a
subset of $W(L^{G_{i}})$ in the obvious way. We denote the set of
original points in $W(L^{G_{i}})$ by $W(L^{G_{i}})^{\orig}$.
\end{defn}
Obviously $\alpha_{ij}(W(L^{G_{i}})^{\orig})=W(L^{G_{j}})^{\orig}$.
We also note that for $1\ne H\vartriangleleft G$ and $i\in[n]$,
we have $H\not\subset G_{i}$ and $L^{\left\langle G_{i},H\right\rangle }\subsetneq L^{G_{i}}$.
Indeed, if a normal subgroup $N$ of $G$ is contained in some $G_{i}$,
then, since $G_{i}$, $i\in[n]$ are conjugate one another, $N$ is
contained in all $G_{i}$. Therefore $N=1$.
\begin{lem}
We have
\[
\psi_{i}(Y^{\circ}(L)^{G})=W(L^{G_{i}})^{\orig}.
\]
\end{lem}
\begin{proof}
Suppose that $\psi_{i}(y)$ is not original for some $i$, and hence
for all $i$. There exists a non-trivial normal subgroup $H$ of $G$
such that for every $i$, $\psi_{i}(y)\in W(L^{\left\langle G_{i},H\right\rangle })$.
Let us write $y=\prod_{i=1}^{n}y_{i}$ with $y_{i}\in W(L)$. For
$h\in H$ and $i\in[n]$, we have 
\[
y_{hi}=y_{i}.
\]
Hence $y\in Y^{H}\subset Y\setminus Y^{\circ}$. 

Conversely, suppose that $y\in Y^{g}(L)$ for some $1\ne g\in G$.
For each $i$, $y_{i}:\Spec L\to W$ factors through $\Spec L^{g}$
and $\Spec L^{G_{i}}$, and hence $\Spec L^{g}\cap L^{G_{i}}$. Since,
for each $i,j\in[n]$, we have the commutative diagram
\[
\xymatrix{\Spec L^{G_{i}}\ar[d]_{\cong}\ar[rr]^{\psi_{i}(x)} &  & W\\
\Spec L^{G_{j}}\ar[urr]_{\psi_{j}(x)} &  & ,
}
\]
each morphism $\psi_{i}(x)$ factors also through $\Spec L^{h}\cap L^{G_{i}}$
for all elements $h$ conjugate to $g$. In consequence, $\psi_{i}(x)$
factors through $\Spec L^{\left\langle G_{i},H\right\rangle }$ with
$H$ the normal closure of $\{g\}$. \end{proof}
\begin{defn}
We put
\begin{align*}
W[K]_{G}: & =\bigsqcup_{L\in\Fie GK}\frac{W(L^{G_{1}})}{Z(G)},\\
W[K]_{G}^{\orig}: & =\bigsqcup_{L\in\Fie GK}\frac{W(L^{G_{1}})^{\orig}}{Z(G)}.
\end{align*}
\end{defn}
\begin{lem}
\label{lem: freeness Z(G)-action}For each $L\in\Fie GK$, the $Z(G)$-action
on $W(L^{G_{1}})^{\orig}$ is free.\end{lem}
\begin{proof}
Let $g\in Z(G)$. By definition, any original point $w:\Spec L^{G_{1}}\to W$
does not factor through $\Spec L^{\left\langle G_{1},g\right\rangle }$.
Namely, if $\bar{w}$ is the image of the unique point of $\Spec L^{G_{1}}$
by the morphism $w$, then the image of 
\[
w^{*}:\cO_{W,\bar{w}}\to L^{G_{1}}
\]
is not contained in $L^{\left\langle G_{1},g\right\rangle }$. This
shows $g\cdot w\ne w$. 
\end{proof}
The maps $\psi_{1}$ for all $L\in\Fie GK$ induce a map
\[
\psi:Y\{K\}\to W[K]_{G}.
\]
Restricting it, we obtain a bijection, which we denote by the same
symbol: 
\[
\psi:Y^{\circ}\{K\}\to W[K]_{G}^{\orig}.
\]

\begin{prop}
\label{prop: H(y)=00003DH(psi(w))}Let $D$ be a metrized Cartier
divisor on $W$ and define a metrized Cartier divisor on $Y$ by 
\[
E:=\sum_{i=1}^{n}p_{i}^{*}D.
\]
For $y\in Y\{K\},$
\[
H_{E}(y)=H_{D}(\psi(y))^{\sharp G}.
\]
\end{prop}
\begin{proof}
Let $\sigma_{i}\in G$ be such that $\sigma_{i}(1)=i$. We have the
following commutative diagram.
\[
\xymatrix{\Spec L\ar[r]^{y\circ\sigma_{i}}\ar[d] & W^{n}\ar[d]^{p_{i}}\\
\Spec L^{G_{1}}\ar[r]_{\psi(y)} & W
}
\]
From Lemma \ref{lem: properties height}, 
\[
H_{p_{i}^{*}D}(y)=H_{p_{i}^{*}D}(y\circ\sigma_{i})=H_{D}(p_{i}\circ y\circ\sigma_{i})=H_{D}(\psi(y))^{\sharp G_{1}}.
\]
Since $\sharp G=n\cdot\sharp G_{1}$,
\[
H_{E}(y)=\prod_{i=1}^{n}H_{p_{i}^{*}D}(y)=H_{D}(\psi(y))^{\sharp G},
\]
as desired.
\end{proof}
For such a divisor $E$ as in the lemma, we define a $\QQ$-Cartier
divisor $\bar{E}$ on $X$ as $\frac{1}{\sharp G}(\pi_{*}E)$ and
metrize it so that $\pi^{*}\bar{E}=E$ as metrized divisors. 
\begin{rem}
Indeed we can always metrize $\bar{E}$ in this way. The divisor $E$
corresponds to a $G$-invertible sheaf $\cL$ on $Y$. Hence the divisor
$\sharp G\cdot\bar{E}=\pi_{*}E$ is a Cartier divisor corresponding
to the $G$-invariant part of $\pi_{*}(\cL{}^{\otimes\sharp G})$,
which naturally inherits the metric on $\cL{}^{\otimes\sharp G}$.
This defines a metric on $\bar{E}$ satisfying the desired property.
\end{rem}
Consider the bijection 
\begin{equation}
\tau:=\psi\circ(\pi_{*})^{-1}:X(K)^{\prim}\to W[K]_{G}^{\orig}.\label{eq: tau}
\end{equation}
From Propositions \ref{prop:H(y)=00003DH(pi(y))} and \ref{prop: H(y)=00003DH(psi(w))},
we obtain the following corollary.
\begin{cor}
\label{cor: H(x)=00003DH(tau(x))}For $x\in X(K)^{\prim}$, 
\[
H_{\bar{E}}(x)=H_{D}(\tau(x)).
\]

\end{cor}

\section{Height zeta functions and possible applications}

\subsection{Height zeta functions}

The height zeta function is the main analytic tool in studies of the
density of rational points, and also in our study of relating it with
the density of number fields. For details of height zeta functions,
we refer the reader to \cite{MR2647601}.
\begin{defn}
Let $S$ be a $K$-variety endowed with a height function $H$. For
a certain set $U$ of points of $S$, we define the \emph{height zeta
function }of $U$ to be
\[
Z_{U}(s):=\sum_{x\in U}H(x)^{-s},
\]
for $s\in\CC$ whenever the series absolutely converges. 
\end{defn}
Let $N_{U}(B):=\sharp\{x\in U\mid H(x)\le B\}$. If $a\in\RR$ is
the abscissa of absolute convergence of $Z_{U}(s)$ and if $Z_{U}(s)$
has a meromorphic continuation to a neighborhood of $s=a$, then a
version of Tauberian theorem implies that 
\[
N_{U}(B)\sim c\cdot B^{a}\cdot(\log B)^{b-1},
\]
where $b$ is the order of the pole of $Z_{U}(s)$ at $s=a$ and 
\[
c=a\cdot(b-1)!\cdot\lim_{s\to a}(s-a)^{b}\cdot Z_{U}(s).
\]

From now on, we suppose that projective spaces $\PP^{m}$ have anti-canonical
divisors metrized as in Section \ref{subsec:Projective-spaces} and
the associated height functions. We need the following precise description
of the height zeta function of a projective space.
\begin{thm}[\cite{Franke:1989go}]
Let $L$ be a number field. The height zeta function $Z_{\PP^{m}(L)}(s)$
absolutely converges for $\Re(s)>1$ and has a meromorphic continuation
to the whole $s$-plane. Moreover it has a simple pole at $s=1$ whose
residue is equal to $\delta_{L,m}$ defined in (\ref{eq: delta def}). 
\end{thm}
In what follows, we suppose that $G$ is a transitive subgroup of
$S_{n}$ and $m\cdot\ind(G)\ge2$. We suppose that the quotient variety
$X:=(\PP^{m})^{n}/G$ is given the metrized anti-canonical divisor
such that we have an equality of metrized divisors, 
\[
\pi^{*}(-K_{X})=\sum p_{i}^{*}(-K_{\PP^{m}}),
\]
where $\pi:(\PP^{m})^{n}\to X$ is the quotient map and $p_{i}:(\PP^{m})^{n}\to\PP^{m}$
is the $i$-th projection. From Lemma \ref{lem: freeness Z(G)-action}
and Corollary \ref{cor: H(x)=00003DH(tau(x))}, we obtain the following
result.
\begin{thm}
\label{thm: zeta equal permutation}We have 
\[
Z_{X(K)^{\prim}}(s)=\frac{1}{\sharp Z(G)}\sum_{L\in\Fie GK}Z_{\PP^{m}(L^{G_{1}})^{\orig}}(s),
\]
whenever the series $Z_{X(K)^{\prim}}(s)$ absolutely converges. 
\end{thm}
The equality can be considered as a \emph{mass formula }of (large
or small) $G$-fields, a formula for counting $G$-fields with weight
$\frac{1}{\sharp Z(G)}\cdot Z_{\PP^{m}(L^{G_{1}})^{\orig}}(s)$. This
is also a global analogue of the wild McKay correspondence \cite{Wood-Yasuda-I,MR3230848,Yasuda:2013fk,Yasuda:2014fk}
in the case of local fields.
\begin{rem}
In the situation of Section \ref{sub:General-Galois-covers}, we similarly
have
\[
Z_{X(K)^{\prim}}(s)=\frac{1}{\sharp Z(G)}\cdot\sum_{L\in\Fie GK}Z_{Y^{\circ}(L)^{G}}(s).
\]
Using this equality instead of the one above, we might be able to
apply arguments below to counting of $G$-fields with respect to various
weights other than powers of discriminants, as proposed in \cite[Section 4.2]{Ellenberg:2005bn}.
For this purpose, we need to know the height zeta function $Z_{Y^{\circ}(L)^{G}}(s)$
as well as we know now the usual height zeta function of a projective
space. 
\end{rem}
Now we show that function $Z_{\PP^{m}(L^{G_{1}})\orig}(s)$ has as
nice properties as $Z_{\PP^{m}(L^{G_{1}})}(s)$ does. 
\begin{defn}
Define the following set of subgroups of $G$
\[
A(G):=\{\left\langle G_{1},H\right\rangle \mid H\vartriangleleft G\},
\]
For $I\in A(G)$, we call a point $x\in W(L^{I})$ \emph{original
}if $x$ is not in the image of $W(L^{J})$ for any $J\in A(G)$ with
$I\subsetneq J$. We denote by $W(L^{I})^{\orig}$ the set of original
$L^{I}$-points. 
\end{defn}
For $I=G_{1}$, this definition of original coincides with Definition
\ref{def: original}. Obviously we have
\[
\PP^{m}(L^{G_{1}})=\bigsqcup_{I\in A(G)}\overline{\PP^{m}(L^{I})^{\orig}},
\]
where $\overline{\PP^{m}(L^{I})^{\orig}}$ is the image of $\PP^{m}(L^{I})^{\orig}$
by the natural injection $\PP^{m}(L^{I})^{\orig}\to\PP^{m}(L^{G_{1}})$.
From Lemma \ref{lem: properties height}, 
\[
Z_{\PP^{m}(L^{G_{1}})}(s)=\sum_{I\in A(G)}Z_{\PP^{m}(L^{I})^{\orig}}([I:G_{1}]\cdot s).
\]
Similarly, for each $I\in A(G)$, 
\[
Z_{\PP^{m}(L^{I})}([I:G_{1}]\cdot s)=\sum_{\substack{J\in A(G)\\
I\subset J
}
}Z_{\PP^{m}(L^{J})^{\orig}}([J:G_{1}]\cdot s).
\]

\begin{prop}
We have 
\begin{align*}
Z_{\PP^{m}(L^{G_{1}})^{\orig}}(s) & =\sum_{I\in A(G)}\mu(I)\cdot Z_{\PP^{m}(L^{I})}([I:G_{1}]\cdot s)\\
 & =Z_{\PP^{m}(L^{G_{1}})}(s)+\sum_{\substack{I\in A(G)\\
I\ne G_{1}
}
}\mu(I)\cdot Z_{\PP^{m}(L^{I})}([I:G_{1}]\cdot s),
\end{align*}
where $\mu(I)=\mu(G_{1},I)\in\ZZ$ is the Möbius function for the
poset $A(G)$.\end{prop}
\begin{proof}
This follows from the Möbius inversion formula (see \cite[Proposition 3.7.2]{MR1442260}).
\end{proof}
The following is a direct consequence of the proposition.
\begin{cor}
For $L\in\Fie GK$, the function $Z_{\PP^{m}(L^{G_{1}})^{\orig}}(s)$
absolutely converges for $\Re(s)>1$, admits a meromorphic continuation
to the whole $s$-plane, and has a simple pole at $s=1$ with residue
$\delta_{L^{G_{1}},m}$. \end{cor}
\begin{example}
For $G=S_{n}$, we have $A(G)=\{G_{1},S_{n}\}$. Hence 
\[
\PP^{m}(L^{G_{1}})^{\orig}=\PP^{m}(L^{G_{1}})\sqcup\overline{\PP^{m}(K)},
\]
and
\[
Z_{\PP^{m}(L^{G_{1}})^{\orig}}=Z_{\PP^{m}(L^{G_{1}})}(s)-Z_{\PP^{m}(K)}(ns).
\]

\end{example}

\subsection{The density of primitive points}

We continue to work in the following situation: $G\subset S_{n}$
is a transitive subgroup, $X=(\PP^{m})^{n}/G$ and $m\cdot\ind(G)\ge2$.
Let us consider the following more specific version of Conjecture
\ref{conj: Manin variant}.
\begin{conjecture}
\label{conj: Manin specific}There exists a positive constant $c$
such that 
\[
N_{X(K)^{\prim}}(B)\sim c\cdot B\cdot(\log B)^{\cd(X)}\quad(B\to\infty).
\]
\end{conjecture}
\begin{rem}
The set of primitive points $X(K)^{\prim}$ is obtained by removing
a thin subset from $X(K)$ (Remark \ref{rem: thin primitive}). Removing
a thin subset in the context of Manin's conjecture is suggested already
in \cite[page 345]{MR2019019}. The whole set $X(K)$ would not satisfy
the same asymptotic formula as above. For instance, if a subgroup
$H\subset S$ is not transitive, then $(\PP^{m})^{n}/H$ has Picard
number $\ge2$ and the images of its $K$-points on $(\PP^{m})^{n}/S_{n}$
are expected to outnumber the primitive $K$-points. This would be
related to the phenomenon that étale extensions outnumber field extensions
(for instance, see \cite[pages 1034-1035]{Bhargava:2005ih}). For,
if we extend bijection (\ref{eq: tau}), general $K$-points of $X$
correspond to $E$-points of $W$ for an étale $K$-algebra $E$ of
degree $n$. 
\end{rem}
We would like to show partial implications between Conjecture \ref{conj: Manin specific}
and Malle's conjecture, by using the equality above of height zeta
functions and assuming that the following auxiliary conjectures hold. 
\begin{conjecture}
\label{conj: merom continuation}The height zeta function $Z_{X(K)^{\prim}}(s)$
absolutely converges for $\Re(s)>1$, admits a meromorphic continuation
to a neighborhood of $s=1$ with a pole at $s=1$.
\end{conjecture}
The first part of this conjecture follows from Conjecture \ref{conj: Manin specific}.
Although assuming Conjecture \ref{conj: merom continuation} might
not be really necessary, it would simplify arguments.

Although we might encounter an unwished issue of limits, it seems
reasonable to expect the following.
\begin{conjecture}
\label{conj: delta sum}The function $Z_{X(K)^{\prim}}(s)$ has a
simple pole at $s=1$ if and only if 
\[
\sum_{F\in\fie GK}\delta_{F,m}<\infty.
\]
If these equivalent conditions hold, then 
\[
\Res_{s=1}Z_{X(K)^{\prim}}(s)=\frac{\sharp N_{S_{n}}(G)}{\sharp C_{S_{n}}(G)\cdot\sharp G}\cdot\sum_{F\in\fie GK}\delta_{F,m}.
\]
\end{conjecture}
\begin{thm}
\label{thm: Malle implies Manin}Let $G\subset S_{n}$ be a transitive
subgroup, $m$ a positive integer with $m\cdot\ind(G)>2$, and $X:=(\PP^{m})^{n}/G$.
Suppose that for these $G$ and $X$, Conjectures \ref{conj: merom continuation},
\ref{conj: delta sum} and \ref{conj: Malle} hold. Then, Conjecture
\ref{conj: Manin specific} holds for $X$.\end{thm}
\begin{proof}
Since $m\cdot\ind(G)>2$, $X$ is terminal, and $\cd(X)=0$. Therefore
it suffices to show that $Z_{X(K)^{\prim}}(s)$ has a simple pole
at $s=1$. From Conjecture \ref{conj: delta sum}, this is equivalent
to $\sum_{F\in\fie GK}\delta_{F,m}<\infty$. Conjecture \ref{conj: Malle}
shows
\[
\limsup_{B\to+\infty}\frac{\log M_{G,K}(B)}{\log B}=\frac{1}{\ind(G)}.
\]
From the lemma below, the Dirichlet series
\begin{equation}
\sum_{F\in\fie GK}|N_{K/\QQ}(D_{F/K})|^{-s}=d_{K}^{-ns}\cdot\sum_{F\in\fie GK}d_{F}^{-s}\label{eq: Dirichlet discriminant}
\end{equation}
has $1/\ind(G)$ as the abscissa of convergence. In particular, the
series converges for $s=m/2$. From Proposition \ref{prop:Siegel-Brauer},
$\sum_{F\in\fie GK}\delta_{F,m}<\infty$, as desired. \end{proof}
\begin{lem}[{\cite[T. I.2.7]{MR0259079}}]
\label{lem: abscissa Dirichlet general}For a Dirichlet series 
\[
f(s)=\sum_{n=1}^{\infty}\frac{a_{n}}{n^{-s}},
\]
put
\[
a:=\limsup_{n\to+\infty}\frac{\log\left(\sum_{m\le n}a_{m}\right)}{\log n}.
\]
If $0<a<\infty$, then $a$ is the abscissa of convergence of $f(s)$. 
\end{lem}
Malle's conjecture is known to hold for abelian groups, for instance.
On the other hand, even for abelian groups, the quotient variety $X=(\PP^{m})^{n}/G$
is not generally rational. Thus, once Conjectures \ref{conj: merom continuation}
and \ref{conj: delta sum} are verified, Theorem \ref{thm: Malle implies Manin}
would give non-trivial affirmative results. 
\begin{rem}
\label{rem: m ind =00003D2}It appears harder, but more interesting
to consider the case where $m\cdot\ind(G)=2$, (equivalently, $X=(\PP^{m})^{n}/G$
is canonical but not terminal). In this case, Conjectures in this
paper suggest that, when taking the limit
\[
\lim_{N\to+\infty}\sum_{\substack{L\in\Fie GK\\
d_{L^{G_{1}}}\le N
}
}Z_{\PP^{m}(L^{G_{1}})^{\orig}}(s),
\]
the order of pole at $s=1$ would jump up from one to 
\[
\cd(X)+1=\beta(G,K)+1.
\]
However the author does not know either how to show this from information
on the density of $G$-fields, or how to get a positive result on
Malle's conjecture from information on $Z_{X(K)^{\prim}}(s)$. 

The condition $m\cdot\ind(G)=2$ is restrictive, although it is satisfied
by the important case where $G=S_{n}$ and $m=2$. It might be an
interesting problem to ask what is a substitute for $(\PP^{m})^{n}/G$
when $m$ is only a rational number. 
\end{rem}

\subsection{The density of number fields}

For general $G$ and $n$, the best known upper bounds for $M_{G,K}(B)$
are the one by Schmidt \cite{MR1330934}, 
\[
M_{G,K}(B)\ll B^{\frac{n+2}{4}},
\]
and the one by Ellenberg and Venkatesh \cite{Ellenberg:2006js},
\[
N_{G,K}(B)\ll(B\cdot C_{1})^{\exp(C_{2}\sqrt{\log n})},
\]
with $C_{1}$ and $C_{2}$ positive constants. For $G=A_{n}$, Larson
and Rolen \cite{Larson:2013jn} showed
\[
N_{A_{n},K}(B)\ll B^{\frac{n^{2}-2}{4(n-1)}}\cdot(\log B)^{2n+1}.
\]
In particular, a bound independent of $n$ has not been obtained yet. 
\begin{thm}
Suppose that Conjectures \ref{conj: Manin specific}, \ref{conj: merom continuation}
and \ref{conj: delta sum} hold for the given transitive subgroup
$G$ and $m=\left\lceil 3/\ind(G)\right\rceil $. Then, for any $\epsilon>0$,
\[
M_{G,K}(B)\ll B^{\frac{1}{2}\left\lceil \frac{3}{\ind(G)}\right\rceil +\epsilon}.
\]
\end{thm}
\begin{proof}
We first note that from known bounds above, 
\[
0<\limsup_{B\to+\infty}\frac{\log M_{G,K}(B)}{\log B}<\infty.
\]
Therefore we can apply Lemma \ref{lem: abscissa Dirichlet general}
to Dirichlet series (\ref{eq: Dirichlet discriminant}). For $m=\left\lceil 3/\ind(G)\right\rceil $,
since $m\cdot\ind(G)>2$, $X=(\PP^{m})^{n}/G$ is a terminal Fano
variety. From Conjectures \ref{conj: Manin specific} and \ref{conj: delta sum},
we have 
\[
\sum_{F\in\fie GK}\delta_{F,m}<\infty.
\]
From Proposition \ref{prop:Siegel-Brauer}, for any $\epsilon>0$,
\[
\sum_{F\in\fie GK}|N_{K/\QQ}(D_{F/K})|^{-m/2-\epsilon}=d_{K}^{-n(m/2+\epsilon)}\cdot\sum_{F\in\fie GK}d_{F}^{-m/2-\epsilon}<\infty.
\]
Hence the abscissa of convergence of the Dirichlet series is at most
$m/2$. From Lemma \ref{lem: abscissa Dirichlet general}, 
\[
\limsup_{B\to+\infty}\frac{\log M_{G,K}(B)}{\log B}\le\frac{m}{2}.
\]
For every $\epsilon>0$, there exists $B_{0}$ such that for every
$B\ge B_{0}$, 
\[
\frac{\log M_{G,K}(B)}{\log B}\le\frac{m}{2}+\epsilon,
\]
equivalently
\[
M_{G,K}(B)\le B^{m/2+\epsilon}.
\]
This shows 
\[
M_{G,K}(B)\ll B^{m/2+\epsilon}.
\]

\end{proof}
When $G=S_{n}$, the theorem gives
\[
M_{S_{n},K}(B)\ll B^{3/2+\epsilon},
\]
and when $G=A_{n}$, 
\[
M_{S_{n},K}(B)\ll B^{1+\epsilon}.
\]
Thus, once Conjectures \ref{conj: Manin specific}, \ref{conj: delta sum}
and \ref{conj: merom continuation} are proved, we would considerably
improve present bounds for general $n$, although these conjectures
might be equally hard to prove.

\bibliographystyle{amsalpha}
\bibliography{/Users/highernash/Dropbox/Math_Articles/mybib}

\end{document}